\documentclass[11pt]{amsart}

\usepackage{amsmath,amssymb,hyperref,mathabx,color,enumerate,chngpage}

\usepackage{graphicx}
\usepackage{mathrsfs}
\usepackage{amsthm}

\usepackage{indentfirst}
\usepackage{bm}
\usepackage{amssymb}
\usepackage{stmaryrd}
\usepackage{enumerate}
\usepackage{amsmath}
\usepackage{commath}
\usepackage{upgreek}
\usepackage{hyperref}
\usepackage{cleveref}
\usepackage{multirow}
\usepackage{tikz}
\usepackage{pdfsync}
\usepackage{doi}

\newtheorem{thm}{Theorem}[section]

\newcommand{\vb}{{\bs{v}}}

\newcommand{\beq} {\begin{equation}}
\newcommand{\eeq} {\end{equation}}
\newcommand{\bdm} {\begin{displaymath}}
\newcommand{\edm} {\end{displaymath}}

\newcommand{\bit}{\begin{itemize}}
\newcommand{\eit}{\end{itemize}}
\newcommand{\bde}{\begin{description}}
\newcommand{\ede}{\end{description}}

\newcommand{\ben}{\begin{enumerate}}
\newcommand{\een}{\end{enumerate}}
\newcommand{\algn}[1]{\begin{align} #1 \end{align}}
\newcommand{\algns}[1]{\begin{align*} #1 \end{align*}}
\newcommand{\mltln}[1]{\begin{multline} #1 \end{multline}}

\newcommand{\barr}{\begin{array}}
\newcommand{\earr}{\end{array}}

\newcommand{\mc}[1]{\mathcal{#1}}
\newcommand{\LRp}[1]{\left( #1 \right)}

\newcommand{\LRa}[1]{\left< #1 \right>}

\newcommand{\ra}{\rightarrow}

\newcommand{\R}{{\mathbb R}}

\renewcommand{\div}{\operatorname{div}}

\newcommand{\esssup}{\operatorname{esssup}}

\newcommand{\bs}{\boldsymbol}

\newtheorem{theorem}{Theorem}[section]

\newtheorem{rmk}[theorem]{Remark}

\title[Mixed method reaction-diffusion with membrane]{Mixed finite element methods for nonlinear reaction-diffusion equations with interfaces }

\begin{document}

\author{Xinran Jin$^1$ \and Jeonghun J. Lee$^2$}
\address{$^{1,2}$ Department of Mathematics, Baylor University, Waco, Texas, USA }

\email{$^1$xinran\_jin1@baylor.edu, $^2$jeonghun\_lee@baylor.edu}

\begin{abstract}
  We develop mixed finite element methods for nonlinear reaction-diffusion equations with interfaces which have Robin-type interface conditions. We introduce the velocity of chemicals as new variables and reformulate the governing equations. The stability of semidiscrete solutions, existence and the a priori error estimates of fully discrete solutions are proved by fixed point theorem and continuous/discrete Gr\"onwall inequalities. Numerical results illustrating our theoretical analysis are included.
\end{abstract}

\date{April, 2023}
\maketitle

\section{Introduction}
The reaction-diffusion equations are widely used to describe the diffusion of chemical substances with their reactions. Therefore, methods to numerically solve reaction-diffusion equations have also been studied for a very long time. Recently, a reaction-diffusion model interacting with other physical/chemical conditions has been actively studied, beyond the simple reaction-diffusion equations. An example of such extended reaction-diffusion equations is the reaction-diffusion model with a thin membrane in the domain. This model can be used to model the cases where a semi-permeable membrane is involved in reaction-diffusion processes of chemicals. The governing equations are a set of reaction-diffusion equations in which interface conditions on thin membranes are involved (cf. \cite{Kedem-Katchalsky:1961,Kleinhans:1998}). Reaction-diffusion equations with such interface conditions have been studied in several previous studies. Well-posedness of partial differential equation models and numerical methods for some diffusion, advection-diffusion, reaction-diffusion equations with possibly nonlinear interface conditions were studied (\cite{Cangiani-Georgoulis-Sabawi:2020,Cangiani-Georgoulis-Jensen:2013,Cangiani-Georgoulis-Jensen:2016,Cai-Ye-Zhang:2011,Ciavolella-Perthame:2021,Li-Su-Wang-Wang:2021,Bathory-Bulicek-Soucek:2020,Bothe-Pierre:2010,Calabro-Zunino:2006,Quarteroni-Veneziani-Zunino:2001,Brera-Jerome-Mori-Sacco:2010,Chaplain-Giverso-Lorenzi-Preziosi:2019,Chen-Zou:1998,Li-Melenk-Wohlmuth-Zou:2010,Plum-Wieners:2003}).

In this paper, we will study mixed finite element methods to solve nonlinear reaction-diffusion equations with interface conditions, particularly, for the models in \cite{Ciavolella-Perthame:2021}. In mixed finite element methods using the dual mixed form of diffusion equations (see, e.g., \cite{Boffi-Brezzi-Fortin-book}), the velocity of each chemical is chosen as additional variable. As is well known, when the mixing method is used, the numerical solutions satisfy local mass conservation without additional post-processing for numerical solutions, and the flux of chemical passing through the membrane is given as a continuous quantity. 
Another advantage of the mixed method is that preconditioners for fast solvers for this type of interface problems, have already been well developed with theoretical basis (cf. \cite{Budisa-Boon-Hu:2020}). 

The paper is organized as follows. In Section~\ref{sec:preliminaries} we introduce definitions, governing equations of the reaction-diffusion equations with membrane structures, and semidiscrete discretization with finite element methods. In Section~\ref{sec:fully-discrete-scheme} we define fully discrete scheme with the Crank--Nicolson method and prove well-posedness of fully discrete solutions for sufficiently small time step sizes. We prove the a priori error estimates of the fully discrete scheme in Section~\ref{sec:error-analysis} and present numerical experiment results in Section~\ref{sec:numerical-experiment}. Conclusions and future research directions will be given in Section~\ref{sec:conclusion}. 


\section{Preliminaries} \label{sec:preliminaries}

%

Let $\Omega$ be a bounded domain in $\mathbb{R}^d$ ($d=2,3$) with Lipschitz continuous polygonal/polyhedral boundary. For finite element discretization we consider a family of triangulations $\{ \mathcal{T}_h\}_{h>0}$ of $\Omega$ with shape-regular triangles/tetrahedra and without hanging nodes. Here $h>0$ is the maximum radius of triangles/tetrahedra in $\mathcal{T}_h$. The $(d-1)$-dimensional simplices in $\mathcal{T}_h$ will be called facets in the paper. 

For $1 \le r \le \infty$, $L^r(\Omega)$ is the Lebesgue space with the norm  
\algns{
  \| v \|_{L^r(\Omega)} =
  \begin{cases}
    \left( \int_{\Omega} |v(x)|^r \, dx \right)^{1/r}, & \text{ if }1 \leq r < \infty, \\
    \esssup_{x \in \Omega} \{ |v(x)| \}, & \text{ if } r = \infty.
  \end{cases}
}
For a subdomain $D \subset \Omega$ with positive $d$-dimensional Lebesgue measure, $L^2(D)$ and $L^2(D;\R^d)$ be the sets of $\R$- and $\R^d$-valued square integrable functions with inner products $\LRp{v, v'}_D := \int_{D} v v' \,d x$ and
$\LRp{\vb, \vb'}_D := \int_{D} \vb \cdot \vb' \,d x$. For an integer $l \ge 0$, $\mc{P}_l(D)$ and $\mc{P}_l(D;\R^d)$ are the spaces of $\R$- and $\R^d$-valued polynomials of degree $\le l$ on $D$. 
In the paper $H^s(D)$, $s \ge 0$, denotes the Sobolev space based
on the $L^2$-norm with $s$-differentiability on $D$. We refer to \cite{Evans-book} for a rigorous definition of $H^s(D)$. The norm on $H^s(D)$ is denoted by $\| \cdot \|_{s,D}$ and $D$ is omitted if $D = \Omega$. 

For $T >0$ and a separable Hilbert space $\mathcal{X}$, let $C^0 ([0, T] ; \mathcal{X})$ denote the set of functions $f : [0, T] \rightarrow \mathcal{X}$ that are continuous
in $t \in [0, T]$. For an integer $m \geq 1$, we define
\begin{equation*}
  C^m ([0, T]; \mathcal{X}) = \{ f \, | \, \partial_t^i f \in C^0([0, T];\mathcal{X}), \, 0 \leq i \leq m \},
\end{equation*}
where $\partial_t^i f$ is the $i$-th time derivative in the
sense of the Fr\'echet derivative in $\mathcal{X}$ (cf.~\cite{Yosida-book}).
For a function $f : [0, T] \rightarrow \mathcal{X}$, the Bochner norm is defined by
\algns{
  \| f \|_{L^r(0, T; \mathcal{X})} =
  \begin{cases}
    \left( \int_0^T \| f(s) \|_{\mathcal{X}}^r ds \right)^{1/r}, \quad 1 \leq r < \infty, \\
    \esssup_{t \in (0, T)} \| f (t) \|_{\mathcal{X}}, \quad r = \infty.
  \end{cases}
}
$W^{k,r}(0, T; \mathcal{X})$ for a non-negative integer $k$ and $1 \leq r \leq \infty$ is defined by  the closure of $C^k ([0, T]; \mathcal{X})$ with the norm $\| f \|_{W^{k,r}(0, T;\mathcal{X})} = \sum_{i=0}^k \| \partial_t^i f \|_{L^r(0, T; \mathcal{X})}$. The semi-norm $\| f \|_{\dot{W}^{k,r}(0, T;\mathcal{X})}$ is defined by $\| f \|_{\dot{W}^{k,r}(0, T;\mathcal{X})} = \| \partial_t^k f \|_{L^r(0, T; \mathcal{X})}$.

For a normed space $\mathcal{X}$ with norm $\| \cdot \|_{\mathcal{X}}$ and functions $f_1, f_2 \in \mathcal{X}$, $\| f_1, f_2 \|_{\mathcal{X}}$ will denote $\| f_1 \|_X + \| f_2 \|_{\mathcal{X}}$, and $\| f_1, f_2, f_3 \|_{\mathcal{X}}$ is defined similarly.

\subsection{Governing equations } \label{subsec:governing-equations}
\begin{figure}
 \begin{center}
 \begin{tikzpicture}[scale=2.5]
     \draw [color=black,fill = white!20!white](0,0) rectangle (2,2);
     \draw [thin,color=black] (1,0.)--(1,2.)
     node[ near start, below, right ]{$\Gamma$};
     \draw[->] (1,1)--(0.7,1) node[midway,above]{$n$};
     \draw [color=black] (0,1)--(0,2)--(2,2)--(2,1);
     \node at (0.5,0.5) {$\Omega_-$};
     \node at (1.5,0.5) {$\Omega_+$};
 \end{tikzpicture}
 \caption{A model domain $\Omega$ with interface $\Gamma$\phantom{aaaaaaaaaaaaaa} }
 \label{fig:model-domain}
 \end{center}
 \end{figure}
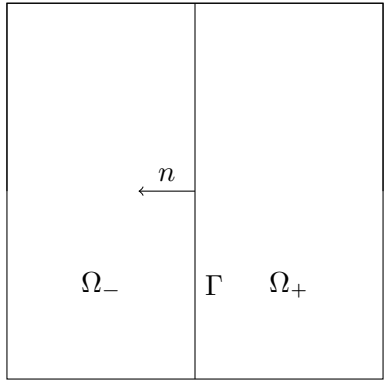

In this subsection we introduce governing equations, a reformulation of the equations, and a variational formulation for finite element methods.

We assume that $\Omega_+, \Omega_- \subset \Omega$ are two disjoint subdomains with polygonal/polyhedral boundaries such that $\overline{\Omega_+} \cup \overline{\Omega_-} = \overline{\Omega}$, and let $\Gamma = \partial \Omega_+ \cap \partial \Omega_-$. For a function $v \in L^2(\Omega)$ such that $v|_{\Omega_j} \in H^1(\Omega_j)$ for $j=+,-$, we use $v|_{\Gamma_j}$ to denote the trace of $v$ on $\Gamma$ from $v|_{\Omega_j}$. Note that $v|_{\Gamma_+} \not = v|_{\Gamma_-}$ in general. Throughout this paper, the unit normal vector field $n$ on $\Gamma$ is the normal vector outward from $\Omega_+$ (see Figure~\ref{fig:model-domain}). 

Suppose that $u_i$, $1\le i \le N$ are real-valued functions on $[0,T] \times \Omega$.
We use $u_i(t)$, $0\le t \le T$, to denote a real-valued function $u_i(t, \cdot)$ defined on $\Omega$.
For given functions 
\begin{align}
    f_i : \mathbb{R}^N \to \mathbb{R}, \quad g_i : [0,T] \times \partial \Omega \to \mathbb{R}
\end{align}
we consider the system of equations to find 
\begin{align*}
    (u_1, \cdots, u_N):[0,T] \times \Omega \to \mathbb{R}^N
\end{align*}
such that 
\begin{subequations}
    \label{eq:u-system-eqs}
\begin{align}
    \partial_t u_i(t) - \div(\kappa_i \nabla u_i(t)) = f_i(u_1(t), \cdots, u_N(t))  & & \text{ in } \Omega,
\end{align}
\end{subequations}
with interface condition
\begin{align}
    -(\kappa_i \nabla u_i(t))\cdot n = K_i(u_{i}|_{\Gamma_+}(t) - u_{i}|_{\Gamma_-}(t)) & & \text{ on } \Gamma, K_i > 0 ,
\end{align}
for all $0< t \le T$, $1 \le i \le N$ and with initial condition 
\begin{align} \label{eq:u-initial-condition}
    (u_1(0), \cdots, u_N(0)) .
\end{align}
To make \eqref{eq:u-system-eqs} a well-posed system of partial differential equations, appropriate boundary conditions are necessary. A set of full Dirichlet boundary conditions 
\begin{align*}
    u_i(t) = g_i(t) \quad \text{ on }\partial\Omega    \quad  \forall 1 \le i \le N, 0 < t \le T,
\end{align*}
can be imposed to make \eqref{eq:u-system-eqs} well-posed. For simplicity, we assume that $g_i = 0$ for $1 \le i \le N$, $0< t\le T$ in the rest of this paper but the discussions below can be extended to more general  boundary conditions including $g_i \not = 0$ and Neumann or mixed boundary conditions on $\partial \Omega$ with appropriate modifications. 
Throughout this paper we assume that the functions $\{f_i\}_{k=1}^N$ satisfy a Lipschitz continuity assumption that as follows: For $v_i, w_i \in L^2(\Omega)$, $1 \le i \le N$, 
\begin{multline} \label{eq:lipschitz-continuity}
    \|f_i(v_1(x), \cdots, v_N(x)) - f_i(w_1(x), \cdots, w_N(x))\| 
    \\
    \le L_i \LRp{ \sum_{i=1}^N |v_i(x) - w_i(x)|^2 }^{\frac 12} 
\end{multline}
for almost every $x \in \Omega$ with a constant $L_i>0$ where $\| \cdot \|$ means the Euclidean norm in $\mathbb{R}^N$.

By introducing $\sigma_i = - \kappa_i \nabla u_i$, we have a system equivalent to \eqref{eq:u-system-eqs} with unknowns $(\sigma_1(t),\cdots ,\sigma_N(t))$, $(u_1(t),\cdots,u_N(t))$ such that
\begin{subequations} \label{eq:mixed-strong-eqs}
\begin{align}
    \label{eq:mixed-strong-eq1}
    \kappa_{i}^{-1} \sigma_i (t) &= -\nabla u_i (t) & & \text{ in } \Omega, 
    \\
    \label{eq:mixed-strong-eq2}
    \partial_t u_i (t)  + \div  \sigma_i (t)  &= f_i(u_1(t) , \cdots , u_N(t) ) & & \text{ in } \Omega
\end{align}
\end{subequations}
with interface conditions
\begin{align}
    \label{eq:mixed-strong-interface}
    \sigma_i(t) \cdot n &= K_i(u_{i}|_{\Gamma_+}(t)  - u_{i}|_{\Gamma_-}(t) ) & & \text{ on } \Gamma     
\end{align}
for all $0 \le t \le T$, $1 \le i \le N$. The boundary conditions 
\begin{align}
    \label{eq:mixed-strong-bc}
    u_i(t)  &= 0 & &  \text{ on } \partial\Omega , \quad 1 \le i \le N, 0 < t \le T 
\end{align}
are imposed as before. 
For initial conditions, in addition to $(u_1(0),\cdots,u_N(0))$ in \eqref{eq:u-initial-condition}, we need 
$(\sigma_1 (0),\cdots ,\sigma_N(0))$ satisfying \eqref{eq:mixed-strong-eq1}, \eqref{eq:mixed-strong-interface} for $t=0$.

To derive a variational formulation of \eqref{eq:mixed-strong-eqs}, let
\begin{align*}
    \Sigma = \{\tau \in H(\div,\Omega): \tau\cdot n|_{\Gamma} \in L^2(\Gamma) \} , \qquad V = L^2(\Omega),
\end{align*}
where $H(\div,\Omega)$ is the subset of $L^2(\Omega; \mathbb{R}^d)$ such that the divergence of $\tau \in L^2(\Omega; \mathbb{R}^d)$ is well-defined as an element in $L^2(\Omega)$. Then, we define $\bs{\Sigma}$ and $\bs{V}$ by 
\algns{
    \bs{\Sigma} = \Sigma_{1} \times \cdots \times \Sigma_{N}, \quad \bs{V} = V_1 \times \cdots \times V_N
}
with $\Sigma_i = \Sigma$, $V_i = V$ for $1 \le i \le N$. Then, after the integration by parts of \eqref{eq:mixed-strong-eq1} for $1 \le i \le N$, we can derive a system of variational equations from \eqref{eq:mixed-strong-eqs} and \eqref{eq:mixed-strong-interface}: Find $(\sigma_1, \cdots , \sigma_N) \in C^0([0,T];\bs{\Sigma})$, $(u_1, \cdots , u_N) \in C^1([0,T];\bs{V})$ such that
\begin{subequations}
    \label{eq:variational-eqs}
\begin{align}
    \label{eq:variational-eq1}
    (\kappa_{i}^{-1} \sigma_i(t), \tau_i)_{\Omega} + \left< K_{i}^{-1} \sigma_i(t) \cdot n, \tau_i \cdot n \right>_{\Gamma} - (u_i(t), \div \tau_i)_{\Omega} &= 0,
    \\
    \label{eq:variational-eq2}
    (\partial_t u_i(t), v_i)_{\Omega} + (\div \sigma_i(t), v_i)_{\Omega} - (f_i (u_1(t), \cdots, u_N(t)), v_i)_{\Omega} &= 0
\end{align}
\end{subequations}
for all $0\le t \le T$, $1\le i \le N$ and for all $(\tau_1, \cdots, \tau_N) \in \bs{\Sigma}$, $(v_1, \cdots, v_N) \in \bs{V}$.

\subsection{Finite element discretization}
In this subsection we present discretization of \eqref{eq:variational-eqs} with finite element methods. 


For an integer $l \geq 0$ and a set $D \subset \R^d$, $\mathcal{P}_l(D)$ is the space of polynomials defined on $D$ of degree at most $l$. Similarly, $\mc{P}_l(D; \mathbb{R}^d)$ is the space of $\mathbb{R}^d$-valued polynomials of degree at most $l$.
For given $l\ge 1$ let us define 
\algn{ \label{eq:Sigmah-local}
  \Sigma_h(T) &= \mc{P}_{l-1}(T; \mathbb{R}^d) +  \begin{pmatrix}
  x_1 \\
  \vdots \\
  x_d
  \end{pmatrix} 
  \mc{P}_{l-1}(T) .
}
Suppose that $\Sigma_{h,i} \subset \Sigma_{i}$ is the Raviart--Thomas(--Nedelec) element (\cite{Nedelec80,RT75,Boffi-Brezzi-Fortin-book}) defined by  
\begin{align*}
  \Sigma_{h,i} &= \{ \tau \in \Sigma_i \,:\, \tau|_T \in \Sigma_h(T), \quad \forall T \in \mc{T}_h \} 
\end{align*}
and $V_h$ is defined by 
\algn{ \label{eq:Vh-local}
    V_h = \{ v \in V \,:\, v|_T \in \mc{P}_{l-1}(T) \quad \forall T \in \mc{T}_h \} .
}
Then, it is well-known that the pair $(\Sigma_{h,i}, V_h)$ satisfies 
%
\begin{align} \label{eq:inf-sup}
    \div \Sigma_{h,i} = V_{h,i}, \qquad 
    \inf_{v_i\in V_{h}} \sup_{\tau_i \in \Sigma_{h,i}} \frac{(v_i,\div \tau_i)_{\Omega}}{\|v_i\| \|\tau_i\|_{\div}} \ge C >0
\end{align}
with a uniform $C>0$ independence of $i$ and mesh sizes of $\mathcal{T}_h$ \cite[p.~406]{Boffi-Brezzi-Fortin-book}. 
%

\subsection{Semidiscrete scheme and stability}
In this subsection we define a semidiscrete scheme of \eqref{eq:variational-eqs} with $\bs{\Sigma}_h \times \bs{V}_h$ and discuss the stability of semidiscrete solutions. For simplicity define $\bs{\sigma}$ and $\bs{u}$ by $(\sigma_1, \sigma_2, \cdots, \sigma_N)$ and $(u_1, \cdots, u_N)$, and semidiscrete solutions $\bs{\sigma}_h :[0,T] \ra \bs{\Sigma}_h$, $\bs{u}_h:[0,T] \ra \bs{V}_h$ are defined similarly.

For 
\begin{align*}
    \bs{\tau} &= (\tau_1,\cdots,\tau_N), \bs{\eta}=(\eta_1,\cdots ,\eta_N) \in \bs{\Sigma}, 
    \\
    \bs{v} &= (v_1,\cdots,v_N), \bs{w} = (w_1,\cdots,w_N) \in \bs{V},
\end{align*}
define three bilinear and one nonlinear forms
\begin{align*}
    a(\bs{\tau},\bs{\eta} ) &:= \sum_{i=1}^{N} (\kappa_{i}^{-1} \tau_i, \eta_i)_{\Omega} +  \sum_{i=1}^N \left<K_{i}^{-1} \tau_i \cdot n, \eta_{i} \cdot n \right>_{\Gamma},
    \\
    b(\bs{\tau}, \bs{v}) &:= \sum_{i=1}^{N} (v_i, \div\tau_i)_{\Omega}, 
    \\
    c(\bs{v},\bs{w}) &:= \sum_{i=1}^{N} (v_i, w_i)_{\Omega},
    \\
    d(\bs{v},\bs{w}) &:= \sum_{i=1}^{N} (f_i(v_i,\cdots,v_N), w_i)_{\Omega} .
\end{align*}
Then, the system \eqref{eq:variational-eqs} can be rewritten as
\begin{subequations}
\begin{align}
    a(\bs{\sigma}(t),\bs{\tau}) - b(\bs{\tau}, \bs{u}(t)) &= 0 & & \forall \bs{\tau} \in \bs{\Sigma}, 
    \\
    b(\bs{\sigma}(t),\bs{v}) + c(\partial_t \bs{u}(t),\bs{v}) - d(\bs{u}(t),\bs{v}) &= 0 & & \forall \bs{v} \in \bs{V}.
\end{align}
\end{subequations}
A discrete-in-space and continuous-in-time semidiscrete scheme with finite element space $\bs{\Sigma}_h \times \bs{V}_h$, is to find $(\bs{\sigma}_h, \bs{u}_h) : [0,T] \to \bs{\Sigma}_h \times \bs{V}_h$ such that
\begin{subequations} \label{eq:semidiscrete-eqs}
\begin{align}
    \label{eq:semidiscrete-eq1}
    a(\bs{\sigma}_h(t),\bs{\tau}) - b(\bs{\tau}, \bs{u}_h(t)) &= 0 & & \forall \bs{\tau} \in \bs{\Sigma}_h, 
    \\
    \label{eq:semidiscrete-eq2}
    b(\bs{\sigma}_h(t), \bs{v}) + c(\partial_t \bs{u}_h(t), \bs{v}) - d(\bs{u}_h(t),\bs{v}) &= 0 & & \forall \bs{v} \in \bs{V}_h
\end{align}    
\end{subequations}
for all $t \in [0,T]$. For stability analysis, let $\bs{\tau} = \bs{\sigma}_h(t)$, $\bs{v} = \bs{u}_h(t)$ and add the equations. Then,
\begin{align*}
    &\frac12 \frac{d}{dt} c(\bs{u}_h(t),\bs{u}_h(t)) + a(\bs{\sigma}_h(t),\bs{\sigma}_h(t)) = d(\bs{u}_h(t),\bs{u}_h(t)).
\end{align*}
By the Lipschitz continuity assumption \eqref{eq:lipschitz-continuity},  we can obtain 
\begin{align*}
    \frac12 \frac{d}{dt} c(\bs{u}_h(t),\bs{u}_h(t)) + a(\bs{\sigma}_h(t), \bs{\sigma}_h(t)) &= d(\bs{u}_h(t),\bs{u}_h(t)) 
    \\
    &\le L \|\bs{u}_h(t)\|_{L^2(\Omega)}^{2} 
\end{align*}
where $L=\max_{1\le i\le N} \{L_i\}$. Recalling that $c(\bs{u}_h(t),\bs{u}_h(t)) = \|\bs{u}_h(t)\|_{L^2(\Omega)}^{2}$, and $a(\bs{\sigma}_h(t),\bs{\sigma}_h(t)) \ge 0$, by Gr\"{o}nwall inequality,
\begin{align*}
    \|\bs{u}_h(t)\|_{L^2(\Omega)} \le e^{2Lt} \|\bs{u}_h(0)\|_{L^2(\Omega)} .
\end{align*}
If $T>0$ is fixed, then 
\begin{align*}
    \max_{0\le t \le T} \|\bs{u}_h(t)\|_{L^2(\Omega)} \le e^{2LT} \|\bs{u}_h(0)\|_{L^2(\Omega)}.
\end{align*}
By \eqref{eq:semidiscrete-eq1}, the definition of $a(\cdot, \cdot)$, and the inf-sup condition \eqref{eq:inf-sup},
\begin{align*}
    \| \bs{\sigma}_h(t) \|_{L^2(\Omega)} \le  a(\bs{\sigma}_h(t), \bs{\sigma}_h(t))^{1/2} \le C\|\bs{u}_h(t)\|_{L^2(\Omega)},
\end{align*}
so we obtain, 
\begin{align*}
    \max_{0\le t \le T} a(\bs{\sigma}_h(t), \bs{\sigma}_h(t))^{1/2} \le Ce^{2LT} \|\bs{u}_h(0)\|_{L^2(\Omega)}  .
\end{align*}

\section{Fully discrete scheme and existence of solutions} 
\label{sec:fully-discrete-scheme}

In this section we present a fully discrete numerical scheme with the Crank--Nicolson method.

For fully discrete scheme, suppose that $(\bs{\sigma}_{h}^{k}, \bs{u}_{h}^{k}) \in \bs{\Sigma}_h \times \bs{V}_h$, a numerical solution of the previous time step is given. The Crank--Nicolson scheme is to find $(\bs{\sigma}_h^{k+1}, \bs{u}_h^{k+1}) \in \bs{\Sigma}_h \times \bs{V}_h$ such that 
\begin{subequations} 
\label{eq:fully-discrete-eqs}
\begin{align}
    \label{eq:fully-discrete-eq1}
    \frac12 a(\bs{\sigma}_{h}^{k}+\bs{\sigma}_{h}^{k+1}, \bs{\tau}) - \frac12 b(\bs{\tau}, \bs{u}_{h}^{k}+\bs{u}_{h}^{k+1}) &= 0,
    \\
    \label{eq:fully-discrete-eq2}
    \frac12 b(\bs{\sigma}_{h}^{k}+\bs{\sigma}_{h}^{k+1} , \bs{v}) + \frac{1}{\Delta t} c\left({\bs{\sigma}_{h}^{k+1} - \bs{\sigma}_{h}^{k}} , v\right)&
    \\
    \notag 
    - \frac12(d(\bs{u}_{h}^{k}, \bs{v})+d(\bs{u}_{h}^{k+1}, \bs{v})) &= 0. 
\end{align}
\end{subequations}
%
Since \eqref{eq:fully-discrete-eqs} is a nonlinear system, existence of $(\bs{\sigma}_h^{k+1}, \bs{u}_h^{k+1})$ is not guaranteed. We use a fixed point theorem to prove existence of $(\bs{\sigma}_h^{k+1}, \bs{u}_h^{k+1})$.

\begin{thm}[Existence and uniqueness of fully discrete solutions]
    \label{thm:existence-uniqueness}
    Suppose that $\Delta t$ is sufficiently small to satisfy 
    \begin{align} \label{eq:small-delta-t}
        L \Delta t < 2
    \end{align}
    where $L>0$ is the constant of Lipschitz continuity of $d(\cdot, \cdot)$ in \eqref{eq:lipschitz-continuity}. Then, there exists a unique $(\bs{\sigma}_h^{k+1}, \bs{u}_h^{k+1}) \in \bs{\Sigma_h} \times \bs{V}_h$ satisfying \eqref{eq:fully-discrete-eqs}.
\end{thm}
\begin{proof}
Recall the fully discrete scheme.
\begin{align*}
    &\frac12 a(\bs{\sigma}_{h}^{k} + \bs{\sigma}_{h}^{k+1}, \bs{\tau}) - \frac12 b(\bs{\tau}, \bs{u}_{h}^{k} + \bs{u}_{h}^{k+1})=0,\\
    &\frac12 b(\bs{\sigma}_{h}^{k}+\bs{\sigma}_{h}^{k+1}, \bs{v}) + c \left(\frac{\bs{u}_{h}^{k+1} - \bs{u}_{h}^{k}}{\Delta t} , \bs{v} \right)\\
    &\quad- \frac12(d(\bs{u}_{h}^{k},\bs{v}) + d(\bs{u}_{h}^{k+1}, \bs{v})) = 0. 
\end{align*}
Assuming that $\bs{\sigma}_{h}^{k}$, $\bs{u}_{h}^{k}$ are given, the system \eqref{eq:fully-discrete-eqs} is to find $(\bs{\sigma}_{h}^{k+1}, \bs{u}_{h}^{k+1})$ such that
\begin{align*}
    &\Delta t(a(\bs{\sigma}_{h}^{k+1},\bs{\tau}) + b(\bs{\tau},\bs{u}_{h}^{k+1}) - b(\bs{\sigma}_{h}^{k+1},\bs{v})) + 2c(\bs{u}_{h}^{k+1},\bs{v}) - \Delta t d(\bs{u}_{h}^{k+1},\bs{v})
    \\
    &\quad=-\Delta t(a(\bs{\sigma}_{h}^{k},\bs{\tau}) + b(\bs{\tau},\bs{u}_{h}^{k} - b(\bs{\sigma}_{h}^{k},\bs{v}))+ 2c(\bs{u}_{h}^{k},\bs{v}) + \Delta t d(\bs{u}_{h}^{k},\bs{v})
    \\
    \\
    &\quad=: G^k(\bs{\tau},\bs{v})
\end{align*}
for all $(\bs{\tau},\bs{v}) \in \bs{\Sigma}_h \times \bs{V}_h$. For simplicity, let $\Phi_{\Delta t}:\bs{\Sigma}_h \times \bs{V}_h \to \bs{\Sigma}_h \times \bs{V}_h$ be a map defined by 
\begin{align*}
    \left<\Phi_{\Delta t}(\bs{\sigma}_h,\bs{u}_h), (\bs{\tau},\bs{v})\right>_{\bs{\Sigma}_h\times \bs{V}_h} = \Delta t(a(\bs{\sigma}_h,\bs{\tau}) + b(\bs{\tau},\bs{u}_h) - b(\bs{\sigma}_h,\bs{v})) + 2c(\bs{u}_h,\bs{v})
\end{align*}
and the above equation can be written by
\begin{align*}
    \left<\Phi_{\Delta t}(\bs{\sigma}_{h}^{k+1},\bs{u}_{h}^{k+1}), (\bs{\tau},\bs{v}) \right>_{\bs{\Sigma}_h\times \bs{V}_h} - \Delta t d(\bs{u}_{h}^{k+1},\bs{v}) = G^k(\bs{\tau},\bs{v})
\end{align*}
Define $(\bs{\sigma}_{h,0}^{k+1},\bs{u}_{h,0}^{k+1})$ by
\begin{align*}
    \left<\Phi_{\Delta t} (\bs{\sigma}_{h,0}^{k+1},\bs{u}_{h,0}^{k+1}), (\bs{\tau},\bs{v}) \right>_{\bs{\Sigma}_h\times \bs{V}_h} = G^k(\bs{\tau},\bs{v}) \quad \forall (\bs{\tau},\bs{v})\in \bs{\Sigma}_h\times \bs{V}_h
\end{align*}
and define $\{(\bs{\sigma}_{h,m}^{k+1},\bs{u}_{h,m}^{k+1})\}_{m=1}^{\infty}$ by
\begin{align*}
    \left<\Phi_{\Delta t} (\bs{\sigma}_{h,m+1}^{k+1},\bs{u}_{h,m+1}^{k+1}), (\bs{\tau},\bs{v}) \right>_{\bs{\Sigma}_h\times \bs{V}_h} - \Delta t d(\bs{u}_{h,m}^{k+1},\bs{v}) = G^k(\bs{\tau},\bs{v}) 
\end{align*}
for all $(\bs{\tau},\bs{v})\in \bs{\Sigma}_h\times \bs{V}_h$ and for $m\ge 0$. By taking difference of the above equation for $m$, $m+1$,
\begin{align*}
    &\left<\Phi_{\Delta t} (\bs{\sigma}_{h,m+1}^{k+1} - \bs{\sigma}_{h,m}^{k+1}, \bs{u}_{h,m+1}^{k+1} - \bs{u}_{h,m}^{k+1}), (\bs{\tau},\bs{v})\right>_{\bs{\Sigma}_h\times \bs{V}_h}\\
    &\quad= \Delta t(d(\bs{u}_{h,m}^{k+1},\bs{v}) - d(\bs{u}_{h,m-1}^{k+1},\bs{v}) )
\end{align*}
for all $(\bs{\tau},\bs{v})\in \bs{\Sigma}_h\times \bs{V}_h$. By Lipschitz continuity of the nonlinearity \eqref{eq:lipschitz-continuity} of $d$,
\begin{align*}
    |d(\bs{u}_{h,m}^{k+1},\bs{v}) - d(\bs{u}_{h,m-1}^{k+1},\bs{v})| \le L\|\bs{u}_{h,m}^{k+1} - \bs{u}_{h,m-1}^{k+1}\|_{L^2(\Omega)} \|\bs{v}\|_{L^2(\Omega)} .
\end{align*}
If $\Delta t$ is small enough to satisfy $\Delta t L < 2$, then
\begin{align*}
    &\left<\Phi_{\Delta t} (\bs{\sigma}_{h,m+1}^{k+1} - \bs{\sigma}_{h,m}^{k+1}, \bs{u}_{h,m+1}^{k+1} - \bs{u}_{h,m}^{k+1}),(\bs{\sigma}_{h,m+1}^{k+1} - \bs{\sigma}_{h,m}^{k+1}, \bs{u}_{h,m+1}^{k+1} - \bs{u}_{h,m}^{k+1}))\right>_{\bs{\Sigma}_h\times \bs{V}_h}\\
    &\quad\le \Delta tL\|\bs{u}_{h,m}^{k+1} - \bs{u}_{h,m-1}^{k+1}\|_{L^2(\Omega)} \|\bs{u}_{h,m+1}^{k+1} - \bs{u}_{h,m}^{k+1}\|_{L^2(\Omega)}\\
    &\quad < 2\|\bs{u}_{h,m}^{k+1} - \bs{u}_{h,m-1}^{k+1}\|_{L^2(\Omega)} \|\bs{u}_{h,m+1}^{k+1} - \bs{u}_{h,m}^{k+1}\|_{L^2(\Omega)} .
\end{align*}
By the definition of $\Phi_{\Delta t}$,
\begin{align*}
    &\left<\Phi_{\Delta t} (\bs{\sigma}_{h,m+1}^{k+1} - \bs{\sigma}_{h,m}^{k+1}, \bs{u}_{h,m+1}^{k+1} - \bs{u}_{h,m}^{k+1}),(\bs{\sigma}_{h,m+1}^{k+1} - \bs{\sigma}_{h,m}^{k+1}, \bs{u}_{h,m+1}^{k+1} - \bs{u}_{h,m}^{k+1}))\right>_{\bs{\Sigma}_h\times \bs{V}_h}\\
    &\quad = \Delta t a(\bs{\sigma}_{h,m+1}^{k+1} - \bs{\sigma}_{h,m}^{k+1},\bs{\sigma}_{h,m+1}^{k+1} - \bs{\sigma}_{h,m}^{k+1}) + 2\|\bs{u}_{h,m+1}^{k+1} - \bs{u}_{h,m}^{k+1}\|_{L^2(\Omega)}^{2}.
\end{align*}

The above inequality and equality imply that $\Phi_{\Delta t}$ is a contraction on $\bs{\Sigma}_h \times \bs{V}_h$ with the norm $\|(\bs{\tau},\bs{v})\|_{\bs{\Sigma}_h\times \bs{V}_h} := (\Delta ta(\bs{\tau},\bs{\tau}) + 2\|\bs{v}\|_{L^2(\Omega)}^{2})^{1/2}$ if $\Delta tL < 2$.
Therefore, there is a unique fixed point $(\bs{\sigma}_{h,\infty}^{k+1}, \bs{u}_{h,\infty}^{k+1}) \in \bs{\Sigma}_h \times \bs{V}_h$ such that 
\begin{align*}
    \left\|\left(\bs{\sigma}_{h,m}^{k+1} - \bs{\sigma}_{h,\infty}^{k+1},\bs{u}_{h,m}^{k+1} - \bs{u}_{h,\infty}^{k+1} \right)\right\|_{\bs{\Sigma}_h\times \bs{V}_h} \to 0 \text{ as } m\to\infty .
\end{align*}
By the Banach contraction principle, this fixed point is unique, so the proof is completed.
\end{proof}

\section{A priori error estimates}
\label{sec:error-analysis}
For $T>0$ let $\Delta t=T/M$ for a natural number $M$ and define $\{t_k\}_{n=0}^{M}$ by $t_k=k\Delta t$. For a variable $g:[0,T] \to X$ for a Hilbert space $X$, we will use $g_{h}^{k}$ and $g^k$ for the numerical solution of $g$ at $t_k$ and $g(t_k)$, respectively. The variable $g$ can be $\bs{\sigma}$, $\bs{u}$ in the problem. 
For simplicity we will also use the definitions
\begin{align*}
    \bar{\partial}_t g^{k+\frac12} := \frac{1}{\Delta t}(g^{k+1} - g^k), \quad g^{k+\frac12} := \frac12(g^k+g^{k+1})
\end{align*}
for any sequence $\{v^k\}_{k=0}^{M}$ of functions with upper index $k$.

Let $\Pi_h: H^1(\Omega; \mathbb{R}^d) \to \Sigma_h$ be the canonical interpolation operator of the Raviart--Thomas element or  the Nedelec $H(\div)$ element of the 1st kind. If $P_h$ is the $L^2$ projection to $V_h$, then $(\Pi_h, P_h)$ satisfies the commuting diagram property
\begin{align} \label{eq:commuting-diagram}
    \div \Pi_h \tau = P_h \div \tau, \qquad \tau \in H^1(\Omega, \mathbb{R}^d) .
\end{align}
On every facet $F$ in $\mathcal{T}_h$ and a normal vector $n_F$ on $F$, 
\algn{ \label{eq:normal-projection}
    \int_F (\tau - \Pi_h \tau) \cdot n_F q \,ds = 0 \quad \forall q \in \mathcal{P}_{l-1}(F) .
}
By extending $\Pi_h$ and $P_h$ to the $N$-copies of $H^1(\Omega;\mathbb{R}^d)$ and $L^2(\Omega)$, we define 
\begin{align*}
    \bs{\Pi}_h : \underbrace{H^1(\Omega; \mathbb{R}^d) \times \cdots \times H^1(\Omega; \mathbb{R}^d)}_{N \text{ tuples}}  \to \bs{\Sigma}_h, \quad \bs{P}_h : \bs{V} \to \bs{V}_h .
\end{align*}
Let
\begin{align}
    \label{eq:e_sigma-def}
    e_{\bs{\sigma}}^{k} &:= {\bs{\sigma}}^k - {\bs{\sigma}}_{h}^{k} = ( {\sigma}_1^k - {\sigma}_{1,h}^{k}, \cdots, {\sigma}_N^k - {\sigma}_{N,h}^{k}) ,
    \\
    \label{eq:e_u-def}
    e_{\bs{u}} &:= {\bs{u}}^k - {\bs{u}}_{h}^{k} = ( u_1^k - u_{1,h}^{k}, \cdots, u_N^k - u_{N,h}^{k}) , 
\end{align}
and define $e_{\bs{\sigma}}^{h,k}, e_{\bs{\sigma}}^{I,k}, e_{\bs{u}}^{h,k}, e_{\bs{u}}^{I,k}$ by
\begin{align*}
    e_{\bs{\sigma}}^{h,k} :={\bs{\Pi}}_{h} {\bs{\sigma}}^k - {\bs{\sigma}}_{h}^{k}, \quad e_{\bs{u}}^{h,k} := {\bs{P}}_{h} {\bs{u}}^k - {\bs{u}}_{h}^{k},
    \\
    e_{\bs{\sigma}}^{I,k} :={\bs{\Pi}}_{h} {\bs{\sigma}}^k - {\bs{\sigma}}_{h}^{k}, \quad e_{\bs{u}}^{I,k} := {\bs{P}}_{h} {\bs{u}}^k - {\bs{u}}^k .
\end{align*}
By a standard approximation theory of interpolation operators, assuming that $\sigma_i^k \in H^r(\Omega; \mathbb{R}^d)$ and $u_i^k \in H^s(\Omega)$ with $r > 1/2$, $s 
\ge 0$,
\algn{ 
    \label{eq:sigma-approx}
    \| {\sigma}_i^k - \Pi_h {\sigma}_i^k  \|_{L^2(\Omega)} &\le Ch^{m} \| \sigma_i^k \|_{H^r(\Omega)} & & \frac 12 < m \le \max\{l, r\}
    \\
    \label{eq:u-approx}
    \| {u}_i^k - P_h u_i^k  \|_{L^2(\Omega)} &\le Ch^{s} \| u_i^k \|_{H^s(\Omega)} & & 0 \le m \le \max\{l, s\} .
}
As immediate extensions, 
\algn{ 
    \label{eq:Pi-approx}
    \| \bs{\sigma}_i^k - \bs{\Pi}_h \bs{\sigma}_i^k  \|_{L^2(\Omega)} &\le Ch^{m} \| \bs{\sigma}^k \|_{H^r(\Omega)} & & \frac 12 < m \le \max\{l, r\}
    \\
    \label{eq:Ph-approx}
    \| \bs{u}^k - \bs{P}_h \bs{u}^k  \|_{L^2(\Omega)} &\le Ch^{s} \| \bs{u}^k \|_{H^s(\Omega)} & & 0 \le m \le \max\{l, s\} .
}
%
By the commuting diagram property \eqref{eq:commuting-diagram} and the property $\div \Sigma_{h} = V_{h}$, 
\begin{subequations}
\label{eq:b-cancellation-eqs}
\begin{align} 
    \label{eq:b-cancellation-eq1}
    b(e_{\bs{\sigma}}^{I,k}, \bs{v}) &= 0 \quad \forall \bs{v} \in \bs{V}_h, 
    \\
    \label{eq:b-cancellation-eq2}
    b(\bs{\tau}, e_{\bs{u}}^{I,k}) &= 0 \quad \forall \bs{\tau} \in \bs{\Sigma}_h .
\end{align}
\end{subequations}
Here we recall a discrete Gr\"{o}nwall inequality before we begin our proof of error estimates (cf. \cite{Heywood-Rannacher:1982,Riviere:book}).
\begin{thm} \label{thm:gronwall-ineq}
Let $\Delta t > 0$, $B, C > 0$ and $\{a_k\}_k$, $\{b_k\}_k$, $\{c_k\}_k$ be sequences of non-negative numbers satisfying 
\begin{align}
    \label{eq:gronwall-condition}
    a_k + \Delta t \sum_{i=0}^{k} b_i \le B + C\Delta t \sum_{i=0}^{k} a_i +  \sum_{i=0}^{k} c_i 
\end{align}
for all $k \ge 0$. Then, if $C\Delta t < 1$,
\begin{align}
    \label{eq:gronwall-conclusion}
    a_k + \Delta t \sum_{i=0}^{k} b_i \le e^{C(k+1)\Delta t} \LRp{B +  \sum_{i=0}^{k} c_i} .
\end{align}
\end{thm}
\begin{rmk}
    We remark that \eqref{eq:gronwall-condition} and \eqref{eq:gronwall-conclusion} are slightly different in \cite{Heywood-Rannacher:1982}. In particular, the summation $\sum_{i=0}^k c_i$ is $\Delta t \sum_{i=0}^{k} c_i$ in \cite{Heywood-Rannacher:1982} but we can show that \eqref{eq:gronwall-condition} implies \eqref{eq:gronwall-conclusion} with the same proof.
\end{rmk}
\begin{thm}
    \label{thm:eh-estimate}
    Suppose that a pair $\bs{\sigma}= (\sigma_1, \cdots, \sigma_N)$, $\bs{u} = (u_1, \cdots, u_N)$ is a solution of \eqref{eq:variational-eqs}. Suppose also that the assumption of Theorem~\ref{thm:existence-uniqueness} holds, and the sequence $\{( \bs{\sigma}_h^k, \bs{u}_h^k)\}_k$ is a solution of  \eqref{eq:fully-discrete-eqs} for given numerical initial data $(\bs{\sigma}_h^0, \bs{u}_h^0) \in \bs{\Sigma}_h \times \bs{V}_h$ satisfying  $a(\bs{\sigma}_{h}^{0}, \bs{\tau}) + b(\bs{\tau}, \bs{u}_{h}^{0}) = 0$.
    %
    %
    Recall the definitions of $e_{\bs{\sigma}}^{h,k}$ and $e_{\bs{u}}^{h,k}$ in \eqref{eq:e_sigma-def}, \eqref{eq:e_u-def}. If $0<\Delta t < C_L$ for $C_L$ depending on $L$, then  
    %
    \algns{
        \notag
    &\|e_{\bs{u}}^{h,k}\|_{L^{2}(\Omega)}^{2} + \frac{\Delta t}{4} \sum_{m=0}^{k-1} a (e_{\bs{\sigma}}^{h,m} + e_{\bs{\sigma}}^{h,m+1}, e_{\bs{\sigma}}^{h,m} + e_{\bs{\sigma}}^{h,m+1})
    \\
    & +a(e_{\bs{\sigma}}^{h,k}, e_{\bs{\sigma}}^{h,k}) + \frac 1{2\Delta t} \sum_{m=0}^{k-1} \| e_{\bs{u}}^{m+1} - e_{\bs{u}}^{h,m} \|_{L^2(\Omega)}^2
    \\
    &\quad \le \|e_{\bs{u}}^{h,0}\|_{L^2(\Omega)}^{2} + a(e_{\bs{\sigma}}^{h,0}, e_{\bs{\sigma}}^{h,0}) 
    \\
    \notag
    &\qquad + C \LRp{ h^{2r} \|\bs{\sigma} , \bs{u}\|_{L^{\infty}(0,t_{k}; H^r(\Omega))}^2  +  (\Delta t)^4 \|\partial_{t}^3 \bs{u}\|_{L^{\infty}(0, t_{k}; L^2(\Omega))}^2 } 
    \\
    \notag
    &\qquad + Ch^{2r}  \| \partial_t \bs{\sigma} \|_{L^{\infty}(0,t_{k}; H^r(\Omega))}^2 + C h^{2r} \sum \| {\bs{u}} \|_{L^\infty(0,t_{k};H^r(\Omega))}^2 
    \\
    }
    for $\frac 12 < r \le l$.
\end{thm}
\begin{proof}
Note that solutions of \eqref{eq:variational-eqs} satisfy
\begin{align*}
    \frac12 a(\bs{\sigma}^k + \bs{\sigma}^{k+1}, \bs{\tau}) - \frac12 b(\bs{\tau}, \bs{u}^{k} + \bs{u}^{k+1}) &= 0 ,
    \\
    \frac12 b(\bs{\sigma}^k + \bs{\sigma}^{k+1}, \bs{v}) + \frac12 c(\partial_t \bs{u}^k + \partial_t \bs{u}^{k+1}, \bs{v}) - \frac12(d(\bs{u}^k, \bs{v}) + d(\bs{u}^{k+1}, \bs{v})) &= 0
\end{align*}
for all $(\bs{\tau}, \bs{v}) \in \bs{\Sigma}_h \times \bs{V}_h$, $k \ge 0$.
The difference of the above equations and \eqref{eq:fully-discrete-eqs} gives
\begin{align*}
    &\frac12 a(e_{\bs{\sigma}}^{k}+ e_{\bs{\sigma}}^{k+1}, \bs{\tau}) - \frac12 b(\bs{\tau}, e_{\bs{u}}^{k} +e_{\bs{u}}^{k+1}) = 0,
    \\
    &\frac12 b(e_{\bs{\sigma}}^{k}+ e_{\bs{\sigma}}^{k+1},\bs{v}) + c \LRp{\frac12(\partial_t \bs{u}^k + \partial_t \bs{u}^{k+1}) - \frac{1}{\Delta t}(\bs{u}_{h}^{k+1} - \bs{u}_{h}^{k}), \bs{v}}
    \\
    &\quad - \frac12(d(\bs{u}^{k+1}, \bs{v}) - d(\bs{u}_{h}^{k+1}, \bs{v}) + d(\bs{u}^k, \bs{v}) - d(\bs{u}_{h}^{k}, \bs{v})) = 0 
\end{align*}
for all $(\bs{\tau}, \bs{v}) \in \bs{\Sigma}_h \times \bs{V}_h$.
Recalling that $e_{\bs{\sigma}}^{k} = e_{\bs{\sigma}}^{h,k} - e_{\bs{\sigma}}^{I,k}$, $e_{\bs{u}}^{k} = e_{\bs{u}}^{h,k} - e_{\bs{u}}^{I,k}$,
\begin{align*}
    &\frac12 a(e_{\bs{\sigma}}^{h,k} + e_{\bs{\sigma}}^{h,k+1}, \bs{\tau}) - \frac12 b(\bs{\tau}, e_{\bs{u}}^{h,k} + e_{\bs{u}}^{h,k+1}) 
    \\
    &\quad = \frac12 a(e_{\bs{\sigma}}^{I,k} + e_{\bs{\sigma}}^{I,k+1}, \bs{\tau}) - \frac12 b(\bs{\tau}, e_{\bs{u}}^{I,k}+e_{\bs{u}}^{I,k+1}),
    \\
    & \frac12 b(e_{\bs{\sigma}}^{h,k}+e_{\bs{\sigma}}^{h,k+1}, \bs{v}) + \frac{1}{\Delta t} c \left(e_{\bs{u}}^{h,k+1} - e_{\bs{u}}^{h,k}, \bs{v} \right)
    \\
    &\quad =c \left(\frac{1}{\Delta t}(\bs{u}^{k+1} - \bs{u}^k) - \frac12(\partial_t \bs{u}^k + \partial_t \bs{u}^{k+1}), \bs{v} \right) + \frac12 b \left(e_{\bs{\sigma}}^{I,k} + e_{\bs{\sigma}}^{I,k+1}, \bs{v} \right)
    \\
    &\qquad - \frac12 \left(d(\bs{u}_{h}^{k+1},\bs{v}) - d(\bs{u}^{k+1}, \bs{v}) + d(\bs{u}_{h}^{k}, \bs{v}) - d(\bs{u}^k, \bs{v}) \right) .
\end{align*}
By \eqref{eq:b-cancellation-eqs}, we can get reduced error equations
\begin{align*}
    &\frac12 a \left(e_{\bs{\sigma}}^{h,k} + e_{\bs{\sigma}}^{h,k+1}, \bs{\tau} \right) - \frac12 b \left(\bs{\tau}, e_{\bs{u}}^{h,k} + e_{\bs{u}}^{h,k+1} \right) = \frac12 a \left(e_{\bs{\sigma}}^{I,k} + e_{\bs{\sigma}}^{I,k+1}, \bs{\tau} \right),
    \\
    & \frac12 b \left(e_{\bs{\sigma}}^{h,k} + e_{\bs{\sigma}}^{h,k+1}, \bs{v} \right) + \frac{1}{\Delta t} c \left(e_{\bs{u}}^{h,k+1} - e_{\bs{u}}^{h,k}, \bs{v} \right)
    \\
    &=c \left(\frac{1}{\Delta t}(\bs{u}^{k+1} - \bs{u}^k) - \frac12(\partial_t \bs{u}^k + \partial_t \bs{u}^{k+1}), \bs{v} \right) 
    \\
    &\quad - \frac12 \left(d(\bs{u}_{h}^{k+1},\bs{v}) - d(\bs{u}^{k+1}, \bs{v}) + d(\bs{u}_{h}^{k}, \bs{v}) - d(\bs{u}^k, \bs{v}) \right).
\end{align*}
Take $\bs{\tau} = e_{\bs{\sigma}}^{h,k+1}+e_{\bs{\sigma}}^{h,k}$, $\bs{v} = e_{\bs{u}}^{h,k+1} + e_{\bs{u}}^{h,k}$ and add the equations and get
\begin{align*}
    &\frac12 a\left(e_{\bs{\sigma}}^{h,k+1}+e_{\bs{\sigma}}^{h,k}, e_{\bs{\sigma}}^{h,k+1}+e_{\bs{\sigma}}^{h,k}\right) + \frac{1}{\Delta t}\left(\|e_{\bs{u}}^{h,k+1}\|_{L^2(\Omega)}^{2} - \|e_{\bs{u}}^{h,k}\|_{L^2(\Omega)}^{2}\right)
    \\
    &\quad = \frac12 a \left(e_{\bs{\sigma}}^{I,k} + e_{\bs{\sigma}}^{I,k+1}, e_{\bs{\sigma}}^{h,k} + e_{\bs{\sigma}}^{h,k+1}\right) 
    \\
    &\qquad + c\left(\frac{1}{\Delta t}(\bs{u}^{k+1} - \bs{u}^k) - \frac12(\partial_t \bs{u}^k + \partial_t \bs{u}^{k+1}),e_{\bs{u}}^{h,k} + e_{\bs{u}}^{h,k+1}\right)
    \\
    &\qquad -\frac12 \left(d(\bs{u}_{h}^{k+1}, e_{\bs{u}}^{h,k} +e_{\bs{u}}^{h,k+1}) - d(\bs{u}^{k+1}, e_{\bs{u}}^{h,k}+e_{\bs{u}}^{h,k+1}) + d(\bs{u}_{h}^{k}, e_{\bs{u}}^{h,k} +e_{\bs{u}}^{h,k+1})\right)
    \\
    &\qquad + \frac12 d(\bs{u}^{k}, e_{\bs{u}}^{h,k}+e_{\bs{u}}^{h,k+1}).
\end{align*}
By multiplying $\Delta t$ and by a simple algebraic computation, 
\mltln{ \label{eq:eu-intm-1}
    \|e_{\bs{u}}^{h,k+1}\|_{L^2(\Omega)}^{2} + \frac{\Delta t}{2} a\left(e_{\bs{\sigma}}^{h,k+1}+e_{\bs{\sigma}}^{h,k}, e_{\bs{\sigma}}^{h,k+1}+e_{\bs{\sigma}}^{h,k} \right)
    \\
    = \|e_{\bs{u}}^{h,k}\|_{L^2(\Omega)}^{2} + \sum_{j=1}^6 I_j^k
}
where
\algn{
    \label{eq:I1}
    \notag
    I_1^k &:= {\frac{\Delta t}{2} a(e_{\bs{\sigma}}^{I,k} + e_{\bs{\sigma}}^{I,k+1}, e_{\bs{\sigma}}^{h,k} + e_{\bs{\sigma}}^{h,k+1})}
    \\
    \notag
    I_2^k &:= {c \left(\bs{u}^{k+1}- \bs{u}^{k} - \frac{\Delta t}{2} \left(\partial_t \bs{u}^k + \partial_t \bs{u}^{k+1} \right), e_{\bs{u}}^{h,k+1} + e_{\bs{u}}^{h,k} \right)}
    \\
    I_3^k &:= {\frac{\Delta t}{2} \left(d\left(\bs{u}_{h}^{k+1}, e_{\bs{u}}^{h,k+1}+e_{\bs{u}}^{h,k} \right) - d\left(\bs{P}_h\bs{u}^{k+1}, e_{\bs{u}}^{h,k+1} + e_{\bs{u}}^{h,k} \right) \right)}
    \\
    \label{eq:I2}
    I_4^k &:= {\frac{\Delta t}{2} \left( d\left(\bs{u}_{h}^{k}, e_{\bs{u}}^{h,k+1}+e_{\bs{u}}^{h,k} \right)- d\left(\bs{P}_h\bs{u}^{k}, e_{\bs{u}}^{h,k+1} + e_{\bs{u}}^{h,k} \right) \right)}
    \\
    \notag
    I_5^k &:= {\frac{\Delta t}{2} \left(d\left(\bs{P}_h\bs{u}^{k+1}, e_{\bs{u}}^{h,k+1} + e_{\bs{u}}^{h,k} \right) - d\left(\bs{u}^{k+1}, e_{\bs{u}}^{h,k+1}+e_{\bs{u}}^{h,k} \right) \right)}
    \\
    \notag
    I_6^k &:= {\frac{\Delta t}{2} \left( d\left(\bs{P}_h\bs{u}^{k}, e_{\bs{u}}^{h,k+1} + e_{\bs{u}}^{h,k} \right)- d\left(\bs{u}^{k}, e_{\bs{u}}^{h,k+1}+e_{\bs{u}}^{h,k} \right)\right)}.
}
If we take the summation of \eqref{eq:eu-intm-1} over $k$, then we can obtain 
\mltln{ \label{eq:eu-intm-2}
    \| e_{\bs{u}}^{h,k} \|_{L^2(\Omega)}^2 + \frac{\Delta t}{2} \sum_{m=0}^{k-1} a\left(e_{\bs{\sigma}}^{h,m+1}+e_{\bs{\sigma}}^{h,m}, e_{\bs{\sigma}}^{h,m+1}+e_{\bs{\sigma}}^{h,m} \right)
    \\
    =\| e_{\bs{u}}^{h,0} \|_{L^2(\Omega)}^2 + \sum_{m=0}^{k-1} \sum_{j=1}^6 I_j^m .
}
By the Lipschitz continuity assumption \eqref{eq:lipschitz-continuity} and the triangle inequality,
\begin{align}
    \label{eq:I1-estm}
    |I_{3}^{m}| &\le 2L\Delta t \|e_{\bs{u}}^{h,m+1}\|_{L^2(\Omega)}  \left(\|e_{\bs{u}}^{h,m}\|_{L^2(\Omega)} + \|e_{\bs{u}}^{h,m+1}\|_{L^2(\Omega)} \right) ,
    \\
    \label{eq:I2-estm}
    |I_{4}^{m}| &\le 2L\Delta t \|e_{\bs{u}}^{h,m}\|_{L^2(\Omega)}  \left(\|e_{\bs{u}}^{h,m}\|_{L^2(\Omega)} + \|e_{\bs{u}}^{h,m+1}\|_{L^2(\Omega)} \right) , 
\end{align}
so 
\algn{ \label{eq:I12-estm}
    |I_3^m + I_4^m| &\le 4 \Delta t L \LRp{\| e_{\bs{u}}^{h,m} \|_{L^2(\Omega)}^2 + \| e_{\bs{u}}^{h,m+1} \|_{L^2(\Omega)}^2 } .
}
By \eqref{eq:lipschitz-continuity}, \eqref{eq:Ph-approx}, the triangle inequality, and Young's inequality, 
\algn{
    \notag 
    |I_{5}^{m} + I_6^m| &\le \Delta t h^r C\left(\|\bs{u}^m\|_{H^r(\Omega)} + \|\bs{u}^{m+1}\|_{H^r(\Omega)}\right) \left(\|e_{\bs{u}}^{h,m} + e_{\bs{u}}^{h,m+1}\|_{L^2(\Omega)}\right) 
    \\
    \label{eq:I34-estm}
    &\le C\Delta t h^{2r} \|\bs{u}\|_{L^{\infty}(t_m,t_{m+1};H^r(\Omega))}^2 
    \\
    \notag 
    &\quad + \frac{\Delta t}{4} \left(\|e_{\bs{u}}^{h,m}\|_{L^2(\Omega)}^2 + \|e_{\bs{u}}^{h,m+1}\|_{L^2(\Omega)}^2 \right) .
}
Note that 
\algns{
    \sum_{i=1}^N \LRa{K_i^{-1} e_{\sigma_i}^{I,k} \cdot n, \tau_i \cdot n } = 0 \quad \forall \bs{\tau} \in \bs{\Sigma}_h
}
by \eqref{eq:normal-projection}. Then, \eqref{eq:Pi-approx}, the Cauchy--Schwarz and Young's inequalities give
\algn{
    \label{eq:I5-estm}
    |I_{1}^{m}| &= {\frac{\Delta t}{2} |a(e_{\bs{\sigma}}^{I,m} + e_{\bs{\sigma}}^{I,m+1}, e_{\bs{\sigma}}^{h,m} + e_{\bs{\sigma}}^{h,m+1})}|
    \\
    \notag 
    &\le  \frac{\Delta t}{2} \|e_{\bs{\sigma}}^{I,m} + e_{\bs{\sigma}}^{I,m+1}\|_{L^2(\Omega)} \| e_{\bs{\sigma}}^{h,m} + e_{\bs{\sigma}}^{h,m+1} \|_{L^2(\Omega)} 
    \\
    \notag
    &\le C\Delta t h^{2r} \|\bs{\sigma}\|_{L^{\infty}(t_m,t_{m+1}; H^r(\Omega))}^2 
    \\
    \notag 
    &\quad + \frac{\Delta t}{4} a (e_{\bs{\sigma}}^{h,m} + e_{\bs{\sigma}}^{h,m+1}, e_{\bs{\sigma}}^{h,m} + e_{\bs{\sigma}}^{h,m+1}).
}
Lastly, we can estimate $I_2^m$ by Cauchy--Schwarz and Young's inequalities,
\algn{
    \notag 
    |I_{2}^{m}|&\le C(\Delta t)^3 \|\partial_{t}^3 \bs{u}\|_{L^{\infty}(t_m, t_{m+1}; L^2(\Omega))} \|e_{\bs{u}}^{h,m} +e_{\bs{u}}^{h,m+1} \|_{L^2(\Omega)}
    \\
    \label{eq:I6-estm}
    &\le C (\Delta t)^5 \|\partial_{t}^3 \bs{u}\|_{L^{\infty}(t_m, t_{m+1}; L^2(\Omega))}^2 
    \\
    \notag
    &\quad + \frac {\Delta t}4 \LRp{ \|e_{\bs{u}}^{h,m} \|_{L^2(\Omega)}^2 + \|e_{\bs{u}}^{h,m+1} \|_{L^2(\Omega)}^2 } .
}
Applying \eqref{eq:I12-estm}, \eqref{eq:I34-estm}, \eqref{eq:I5-estm}, \eqref{eq:I6-estm} to \eqref{eq:eu-intm-2}, we get 
\algn{
    \notag
    &\|e_{\bs{u}}^{h,k}\|_{L^{2}(\Omega)}^{2} + \frac{\Delta t}{4} \sum_{m=0}^{k-1} a (e_{\bs{\sigma}}^{h,m} + e_{\bs{\sigma}}^{h,m+1}, e_{\bs{\sigma}}^{h,m} + e_{\bs{\sigma}}^{h,m+1})
    \\
    \label{eq:eu-intm-estm}
    &\quad \le \|e_{\bs{u}}^{h,0}\|_{L^2(\Omega)}^{2} + \Delta t \LRp{4L^2 + \frac 12 } \sum_{m=0}^{k-1} \LRp{ \|e_{\bs{u}}^{h,m} \|_{L^2(\Omega)}^2 + \|e_{\bs{u}}^{h,m+1} \|_{L^2(\Omega)}^2 } 
    \\
    \notag
    &\qquad + C \Delta t \sum_{m=0}^{k-1} \LRp{ h^{2r} \|\bs{\sigma} , \bs{u}\|_{L^{\infty}(t_m,t_{m+1}; H^r(\Omega))}^2   +  (\Delta t)^4 \|\partial_{t}^3 \bs{u}\|_{L^{\infty}(t_m, t_{m+1}; L^2(\Omega))}^2 } .
}
%
%

Recall that $a(\bs{\sigma}_{h}^{0}, \bs{\tau}) + b(\bs{\tau}, \bs{u}_{h}^{0}) = 0$ as a condition of numerical initial data. Combining this with the fully discrete scheme, we can get 
\algns{
    a(e_{\bs{\sigma}}^{k}, \bs{\tau}) - b(\bs{\tau}, e_{\bs{u}}^{k} ) = 0, \quad \forall k \ge 0.
}
The difference of $k$ and $(k+1)$ time step of the above error equations is 
\algns{
    &\frac12 a(e_{\bs{\sigma}}^{k+1}- e_{\bs{\sigma}}^{k}, \bs{\tau}) - \frac12 b(\bs{\tau}, e_{\bs{u}}^{k+1} -e_{\bs{u}}^{k}) = 0,
}
so we get another set of error equations
\begin{align*}
    &\frac12 a(e_{\bs{\sigma}}^{h,k+1} - e_{\bs{\sigma}}^{h,k}, \bs{\tau}) - \frac12 b(\bs{\tau}, e_{\bs{u}}^{h,k+1} - e_{\bs{u}}^{h,k}) 
    \\
    &\quad = \frac12 a(e_{\bs{\sigma}}^{I,k+1} - e_{\bs{\sigma}}^{I,k}, \bs{\tau}) + \frac12 b(\bs{\tau}, e_{\bs{u}}^{I,k+1}-e_{\bs{u}}^{I,k}),
    \\
    & \frac12 b(e_{\bs{\sigma}}^{h,k}+e_{\bs{\sigma}}^{h,k+1}, \bs{v}) + \frac{1}{\Delta t} c \left(e_{\bs{u}}^{h,k+1} - e_{\bs{u}}^{h,k}, \bs{v} \right)
    \\
    &\quad =c \left(\frac{1}{\Delta t}(\bs{u}^{k+1} - \bs{u}^k) - \frac12(\partial_t \bs{u}^k + \partial_t \bs{u}^{k+1}), \bs{v} \right) - \frac12 b \left(e_{\bs{\sigma}}^{I,k} + e_{\bs{\sigma}}^{I,k+1}, \bs{v} \right)
    \\
    &\qquad - \frac12 \left(d(\bs{u}_{h}^{k+1},\bs{v}) - d(\bs{u}^{k+1}, \bs{v}) + d(\bs{u}_{h}^{k}, \bs{v}) - d(\bs{u}^k, \bs{v}) \right) .
\end{align*}
Again by \eqref{eq:b-cancellation-eqs}, we get reduced error equations
\begin{align*}
    &\frac12 a(e_{\bs{\sigma}}^{h,k+1} - e_{\bs{\sigma}}^{h,k}, \bs{\tau}) - \frac12 b(\bs{\tau}, e_{\bs{u}}^{h,k+1} - e_{\bs{u}}^{h,k}) = \frac12 a(e_{\bs{\sigma}}^{I,k+1} - e_{\bs{\sigma}}^{I,k}, \bs{\tau}) ,
    \\
    & \frac12 b(e_{\bs{\sigma}}^{h,k}+e_{\bs{\sigma}}^{h,k+1}, \bs{v}) + \frac{1}{\Delta t} c \left(e_{\bs{u}}^{h,k+1} - e_{\bs{u}}^{h,k}, \bs{v} \right)
    \\
    &\quad =c \left(\frac{1}{\Delta t}(\bs{u}^{k+1} - \bs{u}^k) - \frac12(\partial_t \bs{u}^k + \partial_t \bs{u}^{k+1}), \bs{v} \right) 
    \\
    &\qquad - \frac12 \left(d(\bs{u}_{h}^{k+1},\bs{v}) - d(\bs{u}^{k+1}, \bs{v}) + d(\bs{u}_{h}^{k}, \bs{v}) - d(\bs{u}^k, \bs{v}) \right) .
\end{align*}
By taking $\bs{\tau} = 2(e_{\bs{\sigma}}^{h,k+1} + e_{\bs{\sigma}}^{h,k})$, $\bs{v} = 2(e_{\bs{u}}^{h,k+1} - e_{\bs{u}}^{h,k})$, and adding these two equations,
\begin{align*}
    & a(e_{\bs{\sigma}}^{h,k+1}, e_{\bs{\sigma}}^{h,k+1}) -  a(e_{\bs{\sigma}}^{h,k},e_{\bs{\sigma}}^{h,k}) + \frac{2}{\Delta t} c(e_{\bs{u}}^{k+1} - e_{\bs{u}}^{h,k} ,e_{\bs{u}}^{h,k+1} - e_{\bs{u}}^{h,k})
    \\
    &\quad = a(e_{\bs{\sigma}}^{I,k+1} - e_{\bs{\sigma}}^{I,k}, e_{\bs{\sigma}}^{h,k+1} + e_{\bs{\sigma}}^{h,k}) 
    \\
    &\qquad + 2c\left( \frac{1}{\Delta t} (\bs{u}^{k+1} - \bs{u}^k) -\frac12 (\partial_t \bs{u}^k + \partial_t \bs{u}^{k+1} ), e_{\bs{u}}^{h,k+1} - e_{\bs{u}}^{h,k} \right)
    \\
    &\qquad - (d(\bs{u}_{h}^{k+1}, e_{\bs{u}}^{h,k+1} - e_{\bs{u}}^{h,k}) - d(\bs{u}^{k+1}, e_{\bs{u}}^{h,k+1} - e_{\bs{u}}^{h,k}) )
    \\
    &\qquad - ( d(\bs{u}_{h}^{k}, e_{\bs{u}}^{h,k+1} - e_{\bs{u}}^{h,k}) -d(\bs{u}^k, e_{\bs{u}}^{h,k+1} - e_{\bs{u}}^{h,k}))
    \\
    &\quad =:J_{1}^{k} + J_{2}^{k} + J_{3}^{k} + J_{4}^{k} .
\end{align*}
Taking the summation of the above equation over $k$, we can get 
\mltln{
    \label{eq:esigma-intm-estm1}
    a(e_{\bs{\sigma}}^{h,k}, e_{\bs{\sigma}}^{h,k}) + \frac 2{\Delta t} \sum_{m=0}^{k-1} c(e_{\bs{u}}^{m+1} - e_{\bs{u}}^{h,m} ,e_{\bs{u}}^{h,m+1} - e_{\bs{u}}^{h,m})
    \\
    = a(e_{\bs{\sigma}}^{h,0}, e_{\bs{\sigma}}^{h,0}) + \sum_{m=0}^{k-1} \LRp{ J_1^m + J_2^m + J_3^m + J_4^m} .
}
By an argument similar to \eqref{eq:I5-estm}, we estimate $J_1^m$ with Young's inequality by
\algn{
    \notag
    |J_{1}^{m}| &\le  \|e_{\bs{\sigma}}^{I,m+1} - e_{\bs{\sigma}}^{I,m}\|_{L^2(\Omega)} \|e_{\bs{\sigma}}^{h,m+1} + e_{\bs{\sigma}}^{h,m}\|_{L^2(\Omega)},
    \\
    \label{eq:J1-estm}
    &\le C\Delta t h^r \| \partial_t \bs{\sigma} \|_{L^{\infty}(t_m,t_{m+1}; H^r(\Omega))}  \|e_{\bs{\sigma}}^{h,m+1} + e_{\bs{\sigma}}^{h,m}\|_{L^2(\Omega)}
    \\
    \notag
    &\le C\Delta t h^{2r} \| \partial_t \bs{\sigma} \|_{L^{\infty}(t_m,t_{m+1}; H^r(\Omega))}^2 + \frac{\Delta t}{4} \LRp{a(e_{\bs{\sigma}}^{h,m+1},e_{\bs{\sigma}}^{h,m+1}) + a(e_{\bs{\sigma}}^{h,m},e_{\bs{\sigma}}^{h,m}) } .
}
For $J_2^m$,
\algns{
    |J_{2}^{m}| &\le 2\left\|\frac{1}{\Delta t} (\bs{u}^{m+1} - \bs{u}^m) - \frac12(\partial_t \bs{u}^m + \partial_t \bs{u}^{m+1}) \right\|_{L^2(\Omega)} \|e_{\bs{u}}^{h,m+1} - e_{\bs{u}}^{h,m}\|_{L^2(\Omega)},
    \\
    & \le 2{\Delta t} \left\|\frac{1}{\Delta t}(\bs{u}^{m+1} - \bs{u}^m) - \frac12(\partial_t \bs{u}^m + \partial_t \bs{u}^{m+1}) \right\|_{L^2(\Omega)}^{2} 
    \\
    &\quad + \frac{1}{2\Delta t} \|e_{\bs{u}}^{h,m+1} - e_{\bs{u}}^{h,m}\|_{L^2(\Omega)}^{2}
    \\
    & \le 2 (\Delta t)^{-1} \left\|(\bs{u}^{m+1} - \bs{u}^m) - \frac{\Delta t}2(\partial_t \bs{u}^m + \partial_t \bs{u}^{m+1}) \right\|_{L^2(\Omega)}^{2} 
    \\
    &\quad + \frac{1}{2\Delta t} \|e_{\bs{u}}^{h,m+1} - e_{\bs{u}}^{h,m}\|_{L^2(\Omega)}^{2}
    \\
    & \le C (\Delta t)^5 \left\|\partial_t^3 \bs{u} \right\|_{L^{\infty}(t_m,t_{m+1};L^2(\Omega))}^{2} + \frac{1}{2\Delta t} \|e_{\bs{u}}^{h,m+1} - e_{\bs{u}}^{h,m}\|_{L^2(\Omega)}^{2} .
}
By \eqref{eq:lipschitz-continuity}, the Cauchy--Schwarz inequality, Young's inequality, and \eqref{eq:u-approx}, 
\algns{
    |J_{3}^{m}| &\le  L\|\bs{u}_{h}^{m+1}-\bs{u}^{m+1}\|_{L^2(\Omega)} \|e_{\bs{u}}^{h,m+1} - e_{\bs{u}}^{h,m}\|_{L^2(\Omega)}
    \\
    &\le  L\| e_{\bs{u}}^{h,m+1} - e_{\bs{u}}^{I,m+1} \|_{L^2(\Omega)} \|e_{\bs{u}}^{h,m+1} - e_{\bs{u}}^{h,m}\|_{L^2(\Omega)}
    \\
    & \le {\Delta t} L^2 \LRp{ \|e_{\bs{u}}^{h,m+1} \|_{L^2(\Omega)}^2+ \|e_{\bs{u}}^{I,m+1} \|_{L^2(\Omega)}^{2}} + \frac{1}{2\Delta t}\|e_{\bs{u}}^{h,m+1} - e_{\bs{u}}^{h,m}\|_{L^2(\Omega)}^{2}
    \\
    & \le {\Delta t} L^2 \|e_{\bs{u}}^{h,m+1} \|_{L^2(\Omega)}^2 + \frac{1}{2\Delta t}\|e_{\bs{u}}^{h,m+1} - e_{\bs{u}}^{h,m}\|_{L^2(\Omega)}^{2} 
    \\
    &\quad + C\Delta t h^{2r} \| {\bs{u}} \|_{L^\infty(t_m,t_{m+1};H^r(\Omega))}^2 .
}
A completely same argument gives
\begin{align*}
    |J_{4}^{m}| &\le {\Delta t} L^2 \|e_{\bs{u}}^{h,m} \|_{L^2(\Omega)}^2 + \frac{1}{2\Delta t}\|e_{\bs{u}}^{h,m+1} - e_{\bs{u}}^{h,m}\|_{L^2(\Omega)}^{2} 
    \\
    &\quad + C\Delta t h^{2r} \| {\bs{u}} \|_{L^\infty(t_m,t_{m+1};H^r(\Omega))}^2 .
\end{align*}
By combining these estimates of $J_{1}^{k},J_{2}^{k},J_{3}^{k},J_{4}^{k}$, we have
%
%
\algn{
    \notag 
    &a(e_{\bs{\sigma}}^{h,k}, e_{\bs{\sigma}}^{h,k}) + \frac 1{4\Delta t} \sum_{m=0}^{k-1} \| e_{\bs{u}}^{m+1} - e_{\bs{u}}^{h,m} \|_{L^2(\Omega)}^2 
    \\
    \notag
    &\le C\Delta t h^{2r} \sum_{m=0}^{k-1} \| \partial_t \bs{\sigma} \|_{L^{\infty}(t_m,t_{m+1}; H^r(\Omega))}^2 
    \\
    \notag 
    &\quad+ \sum_{m=0}^{k-1} \frac{\Delta t}{4} \sum_{m=0}^{k-1} 
     \LRp{a(e_{\bs{\sigma}}^{h,m+1},e_{\bs{\sigma}}^{h,m+1}) + a(e_{\bs{\sigma}}^{h,m},e_{\bs{\sigma}}^{h,m}) } 
    \\
    \label{eq:esigma-intm-estm}    
    &\quad + C (\Delta t)^5 \sum_{m=0}^{k-1} \left\|\partial_t^3 \bs{u} \right\|_{L^{\infty}(t_m,t_{m+1};L^2(\Omega))}^{2} 
    \\
    \notag
    &\quad + {\Delta t} L^2 \sum_{m=0}^{k-1} \LRp{ \|e_{\bs{u}}^{h,m+1} \|_{L^2(\Omega)}^2 + \|e_{\bs{u}}^{h,m} \|_{L^2(\Omega)}^2 } 
    \\
    \notag 
    &\quad + C\Delta t h^{2r} \sum_{m=0}^{k-1} \| {\bs{u}} \|_{L^\infty(t_m,t_{m+1};H^r(\Omega))}^2 .
}
%
%
%
The sum of \eqref{eq:eu-intm-estm} and \eqref{eq:esigma-intm-estm} gives 
\algns{
    \notag
    &\|e_{\bs{u}}^{h,k}\|_{L^{2}(\Omega)}^{2} + \frac{\Delta t}{4} \sum_{m=0}^{k-1} a (e_{\bs{\sigma}}^{h,m} + e_{\bs{\sigma}}^{h,m+1}, e_{\bs{\sigma}}^{h,m} + e_{\bs{\sigma}}^{h,m+1})
    \\
    &+ a(e_{\bs{\sigma}}^{h,k}, e_{\bs{\sigma}}^{h,k}) + \frac 1{2\Delta t} \sum_{m=0}^{k-1} \| e_{\bs{u}}^{m+1} - e_{\bs{u}}^{h,m} \|_{L^2(\Omega)}^2
    \\
    &\quad \le \|e_{\bs{u}}^{h,0}\|_{L^2(\Omega)}^{2} + \Delta t \LRp{5L^2 + \frac 12 } \sum_{m=0}^{k-1} \LRp{ \|e_{\bs{u}}^{h,m} \|_{L^2(\Omega)}^2 + \|e_{\bs{u}}^{h,m+1} \|_{L^2(\Omega)}^2 } 
    \\
    \notag
    &\qquad + C \Delta t \sum_{m=0}^{k-1} \LRp{ h^{2r} \|\bs{\sigma} , \bs{u}\|_{L^{\infty}(t_m,t_{m+1}; H^r(\Omega))}^2  +  (\Delta t)^4 \|\partial_{t}^3 \bs{u}\|_{L^{\infty}(t_m, t_{m+1}; L^2(\Omega))}^2 } 
    \\
    \notag
    &\qquad + C\Delta t h^{2r} \sum_{m=0}^{k-1} \sum\| \partial_t \bs{\sigma} \|_{L^{\infty}(t_m,t_{m+1}; H^r(\Omega))}^2 
    \\
    \notag 
    &\qquad + \frac{\Delta t}{4} \sum_{m=0}^{k-1} \LRp{a(e_{\bs{\sigma}}^{h,m+1},e_{\bs{\sigma}}^{h,m+1}) + a(e_{\bs{\sigma}}^{h,m},e_{\bs{\sigma}}^{h,m}) } 
    \\
    \notag 
    &\qquad + Ck\Delta t h^{2r} \sum_{m=0}^{k-1}\sum \| {\bs{u}} \|_{L^\infty(t_m,t_{m+1};H^r(\Omega))}^2 .
}
We remark that $\Delta t \sum_{m=0}^{k-1} \| g \|_{L^{\infty}(t_m, t_{m+1}; \mathcal{X})} \le k \Delta t \| g \|_{L^{\infty}(0,t_k; \mathcal{X})}$ for a variable $g = \bs{u}, \bs{\sigma}$ and a norm $\mathcal{X}$, and $k \Delta t = T$ at the final time step $k = M$. Thus, this $\Delta t$ in  $\Delta t \sum_{m=0}^{k-1} \| g \|_{L^{\infty}(t_m, t_{m+1}; \mathcal{X})}$ does not give an additional order of convergence. Finally, the conclusion follows if we apply the discrete Gr\"onwall inequality in Theorem~\ref{thm:gronwall-ineq} to the above inequality. 
%
%
\end{proof}

\bigskip

\section{Numerical experiments}
\label{sec:numerical-experiment}
In this section we present numerical experiment results to illustrate that our theoretical error estimates are valid. All numerical experiments are carried out with FEniCS 2019.1.0 (see \cite{fenics-book}). 

For numerical experiments we set $\Omega = [0,1] \times [0,1]$, $\Gamma = \{1/2\} \times [0,1]$, $\Omega_= = [0,1/2] \times [0,1]$, $\Omega_+ = [1/2,1] \times [0,1]$. We use structured meshes such that $\Omega$ is divided by $M \times M$ subsquares and each subsquare is divided into two triangles. In numerical experiments for convergence rates of errors, we compute errors for $M=4,8,16,32,64$. We remark that this $M$ is not necessarily same as the $M$ for time step sizes in Section~\ref{sec:error-analysis}. In the presentation below, we use $h$ for $1/M$.

{\small
\begin{table}[b]
		\begin{tabular}{lc|cc|cc|cc|cc} \hline
			\multirow{2}{*}{} & \multirow{2}{*}{$h_{\max}$} & \multicolumn{2}{c|}{$  \|u_1 - u_{1,h} \|_{L^2(\Omega)} $} & \multicolumn{2}{c|}{$ \|u_2 - u_{2,h} \|_{L^2(\Omega)} $} & \multicolumn{2}{c|}{$ \|\sigma_1 - \sigma_{1,h} \|_{L^2(\Omega)} $} & \multicolumn{2}{c}{$ \|\sigma_2 - \sigma_{2,h} \|_{L^2(\Omega)} $}
			\\
			 & & error & rate & error & rate & error & rate & error & rate \\ \hline \hline 

			\multirow{5}{*}{} & $1/4$  &  8.0723e-02  &  --  &  5.9019e-02  &  --  &  2.1987e-01  &  --  &   4.4314e-01 \\ 
			 & $1/8$  &  4.0090e-02  &  1.01  & 2.8935e-02  &  1.03  & 1.1241e-01  &  0.97  &   2.3087e-01  &  0.94\\ 
			 & $1/16$  &  2.0008e-02  &  1.00  &  1.4365e-02  &  1.01  &  5.6572e-02  &  0.99  &  1.1677e-01  &  0.98\\ 
			 & $1/32$  &  9.9991e-03  &  1.00  &  7.1683e-03  &  1.00  &  2.8336e-02   &  1.00  &  5.8561e-02  &  1.00\\ 
			 & $1/64$  &  4.9989e-03  &  1.00  &  3.5823e-03  &  1.00  & 1.4175e-02  & 1.00  &  2.9303e-02  &  1.00\\ 
\hline 
		\end{tabular} 
	\caption{{\small Convergence results with $\Delta t = h$, the Crank--Nicolson method, and $(RT_0, DG_0)$. } }
    \label{table:RT0}
\end{table}
}
%


{\small
 \begin{table}[b]
 	\begin{center}
 		\begin{tabular}{lc|cc|cc|cc|cc} \hline
 			\multirow{2}{*}{} & \multirow{2}{*}{$h_{\max}$} & \multicolumn{2}{c|}{$  \|u_1 - u_{1,h} \|_{L^2(\Omega)} $} & \multicolumn{2}{c|}{$ \|u_2 - u_{2,h} \|_{L^2(\Omega)} $} & \multicolumn{2}{c|}{$ \|\sigma_1 - \sigma_{1,h} \|_{L^2(\Omega)} $} & \multicolumn{2}{c}{$ \|\sigma_2 - \sigma_{2,h} \|_{L^2(\Omega)} $}
 			\\
 			 & & error & rate & error & rate & error & rate & error & rate\\ \hline \hline 

 			\multirow{5}{*}{} & $1/4$ &  5.2567e-03  &  --  &  1.1226e-02  &  --  &  2.4258e-02 &  --  &  6.9012e-02 &  --   \\
 			 & $1/8$  &  1.3281e-03  &  1.98  &  2.8502e-03  &  1.98  & 6.1853e-03  &  1.97 &  1.7601e-02  &  1.97\\ 
 			 & $1/16$  & 3.3292e-04  &  2.00  &  7.1518e-04  &  1.99  &  1.5604e-03  &  1.99  &  4.4421e-03  &  1.99\\ 
 			 & $1/32$  &  8.3284e-05  &  2.00  &  1.7896e-04  &  2.00  &  3.9199e-04  & 1.99  &  1.1164e-03  &  1.99\\ 
 			 & $1/64$  &  2.0824e-05  &  2.00  &  4.4749e-05  &  2.00 & 9.8255e-05  & 2.00  &  2.7989e-045  &  2.00\\ 
 			 \hline
             \hline
 		\end{tabular}
 	\end{center}
 	\caption{{\small Convergence results with $\Delta t = h$, the Crank--Nicolson method, and $(RT_1, DG_1)$.} } 
    \label{table:RT1}
 \end{table}
}

\begin{figure}[h!] 
\centering
\includegraphics[scale=.42]{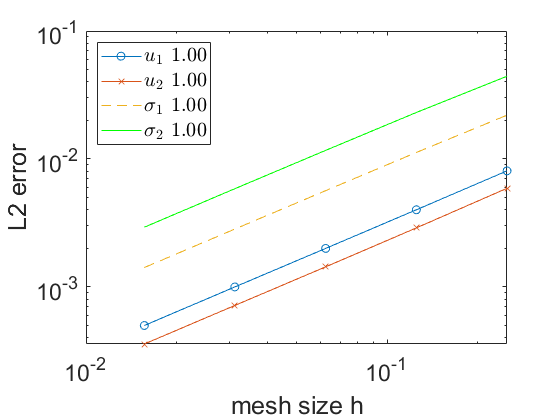} 
\includegraphics[scale=.42]{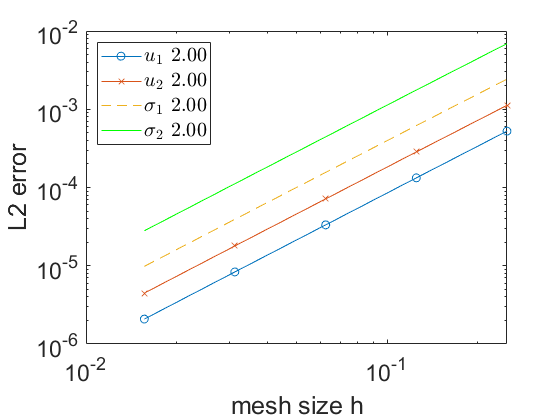} 
\caption{Graphs for asymptotic convergence rates of errors for Table~\ref{table:RT0} and Table~\ref{table:RT1}}
\end{figure}

In our experiments, we used the lowest and the second lowest Raviart--Thomas elements, denoted by $RT_0$ and $RT_1$ for $V_h$. The finite element spaces with piecewise constant and discontinuous piecewise linear polynomials are denoted by $DG_0$ and $DG_1$, and these spaces are used for $V_h$. The stable mixed finite element pairs are $(RT_0, DG_0)$ and $(RT_1, DG_1)$.

In our error analysis, the expected convergence rates of all errors are the first and second orders, respectively. We impose Dirichlet boundary conditions on the top and bottom boundary components of $\Omega$ for $i=1,2$, and impose Neumann boundary conditions on the left and right boundary components of $\Omega$ for $i=1,2$. 

For manufactured solutions we define 
\begin{align*}
    u_i = 
    \begin{cases}
    u_{i,+}, \quad \text{in } \Omega_+,\\
    u_{i,-}, \quad \text{in } \Omega_-,
    \end{cases}
\end{align*}
for $i=1,2$ with appropriate functions $u_{i,\pm}$ which will be given below. First, let 
\algns{
    \phi(x,t) = 1 + (\cos{t})\left(x - \frac12\right)^2,
}
and define
\begin{align*}
    \tilde{u}_{1,-} (x,y) &= \sin{\frac{\pi x}{3}} + \left(x-\frac 12 \right)^2 y (1-y) ,
    \\
    \tilde{u}_{1,+} (x,y) &= \sin{\frac{\pi x}{3}} + 1 + \left(x-\frac 12 \right)^2 \sin (\pi y) ,
    \\
    \tilde{u}_{2,-} (x,y) &= \cos{\frac{\pi x}{3}} + 2 \left(x-\frac 12 \right)^2 y (1-y), 
    \\
    \tilde{u}_{2,+} (x,y) &= \cos \frac{\pi x}{3} - 1 + 2 \left(x-\frac 12 \right)^2 \sin(\pi y). 
\end{align*}
Then, $u_{i,\pm}$, $i=1,2$ are defined by 
\begin{align*}
    {u}_{i,-} = \phi(x,t) u_{i,-}, \quad {u}_{i,+} = \phi(x,t) u_{i,+}.
\end{align*}
For nonlinearities we take $f_1(u_1, u_2) = u_1^2 u_2^3$ and $f_2(u_1, u_2) = u_1^3 u_2^3$. 
Then, $\sigma_{i,\pm}$, $f_{\i,\pm}$, $i=1,2$ are also defined by 
\begin{align*}
    {\sigma}_{i,\pm} &= - \nabla{u}_{i,\pm}, 
    \\
    f_{1,\pm} &= \div {\sigma}_{1,\pm} + u_{1,\pm}^2 u_{2,\pm}^3,
    \\
    f_{2,\pm} &= \div {\sigma}_{2,\pm} + u_{1,\pm}^3 u_{2,\pm}^3.
\end{align*}
%

We remark that these nonlinearities are not Lipschitz continuous with uniform Lipschitz constants in general. However, if $u_1$ and $u_2$ are functions in $L^{\infty}(0,T;L^{\infty}(\Omega))$, then the Lipschitz continuity assumption \eqref{eq:lipschitz-continuity} is satisfied for $0\le t \le T$. Since we use manufactured solutions which are in $L^{\infty}(0,T;L^{\infty}(\Omega))$ in our numerical experiments, our theoretical error estimates are still valid in our numerical experiments.

In Table~\ref{table:RT0} and Table~\ref{table:RT1} we present convergence of errors for $\Delta t = h$ and for $(RT_0, DG_0)$, $(RT_1, DG_1)$ pairs. The results show that optimal convergence rates, which we expected in theoretical analysis, are obtained in all cases.

\section{Conclusion}
\label{sec:conclusion}
In this paper we develop mixed finite element methods for nonlinear reaction-diffusion equations with Robin-type interface conditions on membrane structures in the domain. We proved well-posedness of fully discrete scheme with the Crank--Nicolson method and the a priori error estimates of solutions with a sufficiently small time-step size assumption. In some numerical results, we observed that the errors of solutions converge as expected by our theoretical analysis. In our future research, we will study positivity-preserving numerical methods for the problems.  

\section*{Statements and Declarations}

\noindent {\bf Funding} Jeonghun J. Lee gratefully acknowledge support from the
National Science Foundation (DMS-2110781). 

\bibliographystyle{unsrt}


\end{document}